\documentclass[a4paper,oneside,10pt]{article}%
\usepackage{amsmath}
\usepackage{amsfonts}
\usepackage{amssymb}
\usepackage{graphicx}
\usepackage{color}
\usepackage[square,numbers,sort&compress]{natbib}%
\providecommand{\U}[1]{\protect \rule{.1in}{.1in}}

\newtheorem{theorem}{Theorem}[section]

\newtheorem{definition}[theorem]{Definition}

\newtheorem{proposition}[theorem]{Proposition}
\newtheorem{remark}[theorem]{Remark}

\newenvironment{proof}[1][Proof]{\noindent \textbf{#1.} }{\  $\Box$}
\numberwithin{equation}{section}

\begin{document}

\title{Inequalities for independent random vectors under sublinear expectations}
\author{Xiaojuan Li\thanks{Department of Mathematics, Qilu Normal University, Jinan
250200, China. lxj110055@126.com. }
\and Mingshang Hu \thanks{Zhongtai Securities Institute for Financial Studies,
Shandong University, Jinan, Shandong 250100, China. humingshang@sdu.edu.cn.
Research supported by NSF (No. 12326603, 11671231).} }
\maketitle

\textbf{Abstract}. In this paper, by using the representation theorem for
sublinear expectations, we give a simple proof to obtain two inequalities
about the sample mean for independent random vectors under sublinear expectations.



\addcontentsline{toc}{section}{\hspace*{1.8em}Abstract}

\section{Introduction}

Let $(\Omega,\mathcal{H},\mathbb{\hat{E})}$ be a regular sublinear expectation
space. Then, by Theorem 1.2.2 in \cite{P2019}, there exists a convex set of
probability measures $\mathcal{P}$ on $(\Omega,\sigma(\mathcal{H}))$ such that%
\[
\mathbb{\hat{E}}[X]=\max_{P\in \mathcal{P}}E_{P}[X]\text{ for }X\in
\mathcal{H}.
\]
From this, we can obtain%
\[
-\mathbb{\hat{E}}[-X]=\min_{P\in \mathcal{P}}E_{P}[X]\text{ for }%
X\in \mathcal{H}.
\]
Since $\mathcal{P}$ is convex, we know that $\{E_{P}[X]:P\in \mathcal{P}%
\}=[-\mathbb{\hat{E}}[-X],\mathbb{\hat{E}}[X]]$ by convex combination of
probabilities in $\mathcal{P}$ for each $X\in \mathcal{H}$. Let $\{X_{i}%
\}_{i=1}^{n}$ be an independent random variables in $L^{2}(\Omega)$ under
$\mathbb{\hat{E}}$ (see Definition \ref{new-de-1}). Fang et al. \cite{FP}
first obtained the following inequality about the sample mean (see \cite{HLL}
for different proof method):%
\begin{equation}
\mathbb{\hat{E}}\left[  \left \vert \left(  \frac{1}{n}\sum_{i=1}^{n}%
(X_{i}-\mathbb{\hat{E}}[X_{i}])\right)  ^{+}+\left(  \frac{1}{n}\sum_{i=1}%
^{n}(X_{i}+\mathbb{\hat{E}}[-X_{i}])\right)  ^{-}\right \vert ^{2}\right]
\leq \frac{C}{n}, \label{new-e0-0}%
\end{equation}
where $C$ is a positive constant. If $\mathbb{\hat{E}}$ is linear, then
(\ref{new-e0-0}) is the following classic inequality about the sample mean:%
\begin{equation}
\mathbb{\hat{E}}\left[  \left \vert \frac{1}{n}\sum_{i=1}^{n}(X_{i}%
-\mathbb{\hat{E}}[X_{i}])\right \vert ^{2}\right]  \leq \frac{C}{n}.
\label{new-e0-1}%
\end{equation}
Therefore, (\ref{new-e0-0}) can be viewed as a natural generalization of
(\ref{new-e0-1}) under the sublinear expectation. Let $\{X_{i}\}_{i=1}^{n}$ be
an independent and identically distributed $d$-dimensional random vectors in
$L^{2}(\Omega;\mathbb{R}^{d})$ under $\mathbb{\hat{E}}$. Under the assumption
that the closed convex hull of $\{E_{P}[X_{1}]:P\in \mathcal{P}\}$ is convex
polytope, Fang et al. \cite{FP} obtained the related inequality about the
sample mean in Theorem 2.2.

In this paper, we study the inequalities about the sample mean for independent
$d$-dimensional random vectors under $\mathbb{\hat{E}}$. Let $X_{i}=(X_{i}%
^{1}$,$\ldots$,$X_{i}^{d})^{T}$, $i\leq n$, be an independent $d$-dimensional
random vectors with $d>1$. Since $E_{P}[X_{i}]\in \mathbb{R}^{d}$ for
$P\in \mathcal{P}$ and $\mathbb{R}^{d}$ is partially ordered, it is difficult
to give the form of $\{E_{P}[X_{i}]:P\in \mathcal{P}\}$ by convex combination
of probabilities in $\mathcal{P}$. By the representation theorem for
$\mathbb{\hat{E}}$ and the separation theorem for convex sets, we show that
$\{E_{P}[X_{i}]:P\in \mathcal{P}\}=\Theta_{i}$, where $\Theta_{i}$ is defined
in (\ref{new-e2-3}). Based on this, we give a simple method to get the
inequality about the sample mean without the convex polytope assumption for
$\Theta_{i}$. Furthermore, by Sion's minimax theorem and Lusin's theorem, we
obtain
\[
\inf_{\xi \in L^{2}(\Omega;\Theta)}\mathbb{\hat{E}}\left[  \left \vert \frac
{1}{n}\sum_{i=1}^{n}X_{i}-\xi \right \vert ^{2}\right]  \leq \frac{\bar{\sigma
}_{n}^{2}}{n},
\]
where $\Theta$ and $\bar{\sigma}_{n}^{2}$ are defined in Theorem \ref{new-m2},
$L^{2}(\Omega;\Theta)=\{ \xi \in L^{2}(\Omega;\mathbb{R}^{d}):\xi(\omega
)\in \Theta$ for $\omega \in \Omega \}$.

This paper is organized as follows. We recall some basic results of sublinear
expectation in Section 2. In Section 3, we give three main results for
independent random vectors under regular sublinear expectation.

\section{Preliminaries}

In this section, we recall some basic results of sublinear expectation. The
readers may refer to \cite{P2019} for more details.

Let $\Omega$ be a complete separable metric space. Set $\mathcal{H}%
=C_{b.Lip}(\Omega)$, where $C_{b.Lip}(\Omega)$ denotes the space of bounded
Lipschitz functions on $\Omega$. A functional $\mathbb{\hat{E}}:\mathcal{H}%
\rightarrow \mathbb{R}$ is called a regular sublinear expectation if it satisfies:

\begin{description}
\item[(i)] Monotonicity: if $X\geq Y$, then $\mathbb{\hat{E}}[X]\geq
\mathbb{\hat{E}}[Y]$.

\item[(ii)] Constant preserving: $\mathbb{\hat{E}}[c]=c$ for each
$c\in \mathbb{R}$.

\item[(iii)] Subadditivity: $\mathbb{\hat{E}}[X+Y]$ $\leq \mathbb{\hat{E}%
}[X]+\mathbb{\hat{E}}[Y]$.

\item[(iv)] Positive homogeneity: $\mathbb{\hat{E}}[\lambda X]=\lambda
\mathbb{\hat{E}}[X]$ for each $\lambda \geq0$.

\item[(v)] Regular property: if $X_{n}\downarrow0$, then $\mathbb{\hat{E}%
}[X_{n}]\downarrow0$.
\end{description}

By the Daniell-Stone theorem, we have the following representation theorem for
$\mathbb{\hat{E}}$ (see \cite{DHP11, HP09}).

\begin{theorem}
\label{new-re-1}There exists a convex and weakly compact set of probability
measures $\mathcal{P}$ on $(\Omega,\mathcal{B}(\Omega))$ such that%
\begin{equation}
\mathbb{\hat{E}}[X]=\sup_{P\in \mathcal{P}}E_{P}[X]\text{ for }X\in \mathcal{H},
\label{new-eq-1}%
\end{equation}
where $\mathcal{B}(\Omega)$ is the Borel $\sigma$-field.
\end{theorem}

For each fixed $p\geq1$, denote by $L^{p}(\Omega)$ the completion of
$\mathcal{H}$ under the norm $||X||_{p}:=(\mathbb{\hat{E}}[|X|^{p}])^{1/p}$.
By H\"{o}lder's inequality, we have $L^{p}(\Omega)\subset L^{1}(\Omega)$ for
$p\geq1$. $\mathbb{\hat{E}}$ can be continuously extended to $L^{1}(\Omega)$
under the norm $||\cdot||_{1}$, the representation (\ref{new-eq-1}) still
holds for $X\in L^{1}(\Omega)$ and $\mathbb{\hat{E}}$ is still a regular
sublinear expectation on $L^{1}(\Omega)$. By the Stone-Weierstrass theorem, we
can prove that $C_{b}(\Omega)\subset L^{1}(\Omega)$ (see Lemma 51 and Theorem
52 in \cite{DHP11}), where $C_{b}(\Omega)$ denotes the space of bounded and
continuous functions on $\Omega$.

A $d$-dimensional random vector $X=(X^{1}$,$\ldots$,$X^{d})^{T}$ in
$L^{p}(\Omega;\mathbb{R}^{d})$ means that $X^{i}\in L^{p}(\Omega)$ for $i\leq
d$. For each fixed $X\in L^{1}(\Omega;\mathbb{R}^{d})$, define $\mathbb{\hat
{F}}_{X}:C_{b.Lip}(\mathbb{R}^{d})\rightarrow \mathbb{R}$ as follows:%
\[
\mathbb{\hat{F}}_{X}[\varphi]:=\mathbb{\hat{E}}[\varphi(X)]\text{ for }%
\varphi \in C_{b.Lip}(\mathbb{R}^{d}).
\]
$\mathbb{\hat{F}}_{X}$ is called the distribution of $X$. Two $d$-dimensional
random vectors $X$ and $Y$ in $L^{1}(\Omega;\mathbb{R}^{d})$ are identically
distributed if $\mathbb{\hat{F}}_{X}=\mathbb{\hat{F}}_{Y}$, denoted by
$X\overset{d}{=}Y$. A random vector $Y\in L^{1}(\Omega;\mathbb{R}^{d_{2}})$ is
independent of another random vector $X\in L^{1}(\Omega;\mathbb{R}^{d_{1}})$
means that for each $\psi \in C_{b.Lip}(\mathbb{R}^{d_{1}+d_{2}})$, we have%
\[
\mathbb{\hat{E}}[\psi(X,Y)]=\mathbb{\hat{E}}\left[  \mathbb{\hat{E}}%
[\psi(x,Y)]_{x=X}\right]  .
\]

\begin{definition}
\label{new-de-1}A sequence of random vectors $\{X_{i}\}_{i=1}^{\infty}\subset
L^{1}(\Omega;\mathbb{R}^{d})$ is called independent if $X_{i+1}$ is
independent of $(X_{1}$,$\ldots$,$X_{i})$ for $i\geq1$. Similarly,
$\{X_{i}\}_{i=1}^{n}\subset L^{1}(\Omega;\mathbb{R}^{d})$ with $n>1$ is called
independent if $X_{i+1}$ is independent of $(X_{1}$,$\ldots$,$X_{i})$ for
$1\leq i\leq n-1$.
\end{definition}

The proof of the following proposition is the same as that of Proposition 2.1
in \cite{HLL}.

\begin{proposition}
\label{new-pro-2}Let $\{X_{i}\}_{i=1}^{n}$ be an independent random vectors in
$L^{2}(\Omega;\mathbb{R}^{d})$ under $\mathbb{\hat{E}}$. Then, for each
$P\in \mathcal{P}$ and $\varphi \in C_{Lip}(\mathbb{R}^{d})$, we have%
\[
E_{P}[\varphi(X_{i})|\mathcal{F}_{i-1}]\leq \mathbb{\hat{E}}[\varphi
(X_{i})],\text{ }P\text{-a.s., for }i\leq n,
\]
where $\mathcal{P}$ is given in Theorem \ref{new-re-1}, $C_{Lip}%
(\mathbb{R}^{d})$ denotes the space of Lipschitz functions on $\mathbb{R}^{d}%
$, $\mathcal{F}_{i}=\sigma(X_{1}$,$\ldots$,$X_{i})$ for $i\geq1$ and
$\mathcal{F}_{0}=\{ \emptyset,\Omega \}$.
\end{proposition}

\section{Main results}

Let $\{X_{i}\}_{i=1}^{n}$ be an independent random vectors in $L^{2}%
(\Omega;\mathbb{R}^{d})$ under $\mathbb{\hat{E}}$. For each $i\leq n$, define
$g_{i}:\mathbb{R}^{d}\rightarrow \mathbb{R}$ as follows:%
\begin{equation}
g_{i}(p):=\mathbb{\hat{E}}[\langle p,X_{i}\rangle]\text{ for }p\in
\mathbb{R}^{d}. \label{new-e2-1}%
\end{equation}
It is easy to check that%
\[
g_{i}(p_{1}+p_{2})\leq g_{i}(p_{1})+g_{i}(p_{2})\text{ and }g_{i}(\lambda
p)=\lambda g_{i}(p)\text{ for }\lambda \geq0.
\]
By Theorem 1.2.1 in \cite{P2019} and the separation theorem for convex sets,
we can prove that there exists a unique convex and compact set $\Theta
_{i}\subset \mathbb{R}^{d}$ satisfying%
\begin{equation}
g_{i}(p)=\sup_{\theta \in \Theta_{i}}\langle p,\theta \rangle \text{ for }%
p\in \mathbb{R}^{d}. \label{new-e2-2}%
\end{equation}
Moreover, we have%
\begin{equation}
\Theta_{i}=\{ \theta \in \mathbb{R}^{d}:\langle \theta,p\rangle \leq
g_{i}(p)\text{ for }p\in \mathbb{R}^{d}\}. \label{new-e2-3}%
\end{equation}

The following theorem is our first main result.

\begin{theorem}
\label{new-m1}Let $\{X_{i}\}_{i=1}^{n}$ be an independent random vectors in
$L^{2}(\Omega;\mathbb{R}^{d})$ under $\mathbb{\hat{E}}$. Then we have

\begin{description}
\item[(1)] $\Theta_{i}=\{E_{P}[X_{i}]:P\in \mathcal{P}\}$ for $i\leq n$, where
$\mathcal{P}$ is given in Theorem \ref{new-re-1} and $\Theta_{i}$ is defined
in (\ref{new-e2-3}).

\item[(2)] for each $P\in \mathcal{P}$ and $i\leq n$, $E_{P}[X_{i}%
|\mathcal{F}_{i-1}]\in \Theta_{i}$, $P$-a.s., where $\mathcal{F}_{i}%
=\sigma(X_{1}$,$\ldots$,$X_{i})$ for $i\geq1$ and $\mathcal{F}_{0}=\{
\emptyset,\Omega \}$.
\end{description}
\end{theorem}

\begin{proof}
(1) It follows from Theorem \ref{new-re-1} that%
\[
g_{i}(p)=\mathbb{\hat{E}}[\langle p,X_{i}\rangle]=\sup_{P\in \mathcal{P}}%
E_{P}[\langle p,X_{i}\rangle]=\sup_{P\in \mathcal{P}}\langle p,E_{P}%
[X_{i}]\rangle.
\]
Set $\tilde{\Theta}_{i}=\{E_{P}[X_{i}]:P\in \mathcal{P}\}$ for $i\leq n$. Since
$\mathcal{P}$ is convex and weakly compact, it is easy to verify that
$\tilde{\Theta}_{i}$ is a convex and compact set in $\mathbb{R}^{d}$. Thus we
obtain $\Theta_{i}=\tilde{\Theta}_{i}$ by the uniqueness of the representation
for $g_{i}$.

(2) Set $\mathbb{Q}^{d}=\{p=(p^{1},\ldots,p^{d})^{T}:p^{i}\in \mathbb{Q}$ for
$i\leq d\}$, where $\mathbb{Q}$ is the set of all rational numbers in
$\mathbb{R}$. For each $P\in \mathcal{P}$ and $i\leq n$, it follows from
Proposition \ref{new-pro-2} that for each $p\in \mathbb{Q}^{d}$, we have%
\[
\langle p,E_{P}[X_{i}|\mathcal{F}_{i-1}]\rangle=E_{P}[\langle p,X_{i}%
\rangle|\mathcal{F}_{i-1}]\leq \mathbb{\hat{E}}[\langle p,X_{i}\rangle
]=g_{i}(p),\text{ }P\text{-a.s.}%
\]
Since $\mathbb{Q}^{d}$ is countable and dense in $\mathbb{R}^{d}$, we obtain%
\[
\langle p,E_{P}[X_{i}|\mathcal{F}_{i-1}]\rangle \leq g_{i}(p)\text{ for all
}p\in \mathbb{R}^{d},\text{ }P\text{-a.s.,}%
\]
which implies $E_{P}[X_{i}|\mathcal{F}_{i-1}]\in \Theta_{i}$, $P$-a.s., by
(\ref{new-e2-3}).
\end{proof}

Based on the above theorem, we can get the following two main results.

\begin{theorem}
\label{new-m2}Let $\{X_{i}\}_{i=1}^{n}$ be an independent random vectors in
$L^{2}(\Omega;\mathbb{R}^{d})$ under $\mathbb{\hat{E}}$. Then%
\begin{equation}
\mathbb{\hat{E}}\left[  \rho_{\Theta}^{2}\left(  \frac{1}{n}\sum_{i=1}%
^{n}X_{i}\right)  \right]  \leq \frac{\bar{\sigma}_{n}^{2}}{n},
\label{new-e2-4}%
\end{equation}
where $\Theta=\{ \frac{1}{n}\sum_{i=1}^{n}\theta_{i}:\theta_{i}\in \Theta_{i}$
for $i\leq n\}$, $\Theta_{i}$ is defined in (\ref{new-e2-3}), $\rho_{\Theta
}(x)=\inf \{|x-\theta|:\theta \in \Theta \}$ for $x\in \mathbb{R}^{d}$ and
\begin{equation}
\bar{\sigma}_{n}^{2}=\sup_{i\leq n}\inf_{\theta_{i}\in \Theta_{i}}%
\mathbb{\hat{E}}[|X_{i}-\theta_{i}|^{2}]. \label{new-e2-6}%
\end{equation}

\end{theorem}

\begin{proof}
It follows from Theorem \ref{new-re-1} that%
\begin{equation}
\mathbb{\hat{E}}\left[  \rho_{\Theta}^{2}\left(  \frac{1}{n}\sum_{i=1}%
^{n}X_{i}\right)  \right]  =\sup_{P\in \mathcal{P}}E_{P}\left[  \rho_{\Theta
}^{2}\left(  \frac{1}{n}\sum_{i=1}^{n}X_{i}\right)  \right]  .
\label{new-e2-5}%
\end{equation}
For each $P\in \mathcal{P}$, we have $\frac{1}{n}\sum_{i=1}^{n}E_{P}%
[X_{i}|\mathcal{F}_{i-1}]\in \Theta$, $P$-a.s., by (2) in Theorem \ref{new-m1}.
Thus we obtain%
\begin{align*}
E_{P}\left[  \rho_{\Theta}^{2}\left(  \frac{1}{n}\sum_{i=1}^{n}X_{i}\right)
\right]   &  \leq E_{P}\left[  \left \vert \frac{1}{n}\sum_{i=1}^{n}X_{i}%
-\frac{1}{n}\sum_{i=1}^{n}E_{P}[X_{i}|\mathcal{F}_{i-1}]\right \vert
^{2}\right] \\
&  =\frac{1}{n^{2}}\sum_{i=1}^{n}E_{P}[|X_{i}-E_{P}[X_{i}|\mathcal{F}%
_{i-1}]|^{2}]\\
&  \leq \frac{1}{n^{2}}\sum_{i=1}^{n}E_{P}[|X_{i}-E_{P}[X_{i}]|^{2}]\\
&  =\frac{1}{n^{2}}\sum_{i=1}^{n}\inf_{\theta_{i}\in \Theta_{i}}E_{P}%
[|X_{i}-\theta_{i}|^{2}]\\
&  \leq \frac{1}{n^{2}}\sum_{i=1}^{n}\inf_{\theta_{i}\in \Theta_{i}}%
\mathbb{\hat{E}}[|X_{i}-\theta_{i}|^{2}]\\
&  \leq \frac{\bar{\sigma}_{n}^{2}}{n}%
\end{align*}
for each $P\in \mathcal{P}$. Therefore, we deduce (\ref{new-e2-4}) by
(\ref{new-e2-5}).
\end{proof}

\begin{remark}
Even if $\Omega$ is not a complete separable metric space and $\mathbb{\hat
{E}}$ does not satisfy the regular property (v), the inequality
(\ref{new-e2-4}) still holds. Because we can consider the distribution
$\mathbb{\hat{F}}_{X}$ of $X=(X_{1}$,$\ldots$,$X_{n})$, which is a regular
sublinear expectation on $(\mathbb{R}^{d\times n},C_{b.Lip}(\mathbb{R}%
^{d\times n}))$. If $X_{i}\overset{d}{=}X_{1}$ for $i\leq n$, then
$\Theta=\Theta_{1}$ and $\bar{\sigma}_{n}^{2}=\bar{\sigma}_{1}^{2}%
=\inf_{\theta \in \Theta_{1}}\mathbb{\hat{E}}[|X_{1}-\theta_{1}|^{2}]$ in
Theorem \ref{new-m2}. In particular, we do not need to assume that $\Theta$ is
a convex polytope as Theorem 2.2 in \cite{FP}.
\end{remark}

\begin{theorem}
\label{new-m3}Let $\{X_{i}\}_{i=1}^{n}$ be an independent random vectors in
$L^{2}(\Omega;\mathbb{R}^{d})$ under $\mathbb{\hat{E}}$. Set $L^{2}%
(\Omega;\Theta)=\{ \xi \in L^{2}(\Omega;\mathbb{R}^{d}):\xi(\omega)\in \Theta$
for $\omega \in \Omega \}$, where $\Theta$ is defined in Theorem \ref{new-m2}.
Then%
\begin{equation}
\inf_{\xi \in L^{2}(\Omega;\Theta)}\mathbb{\hat{E}}\left[  \left \vert \frac
{1}{n}\sum_{i=1}^{n}X_{i}-\xi \right \vert ^{2}\right]  \leq \frac{\bar{\sigma
}_{n}^{2}}{n}, \label{new-e2-10}%
\end{equation}
where $\bar{\sigma}_{n}^{2}$ is defined in (\ref{new-e2-6}).
\end{theorem}

\begin{proof}
It follows from Theorem \ref{new-re-1} that%
\begin{equation}
\inf_{\xi \in L^{2}(\Omega;\Theta)}\mathbb{\hat{E}}\left[  \left \vert \frac
{1}{n}\sum_{i=1}^{n}X_{i}-\xi \right \vert ^{2}\right]  =\inf_{\xi \in
L^{2}(\Omega;\Theta)}\sup_{P\in \mathcal{P}}E_{P}\left[  \left \vert \frac{1}%
{n}\sum_{i=1}^{n}X_{i}-\xi \right \vert ^{2}\right]  . \label{new-e2-7}%
\end{equation}
Since $\Theta$ is convex and compact, it is easy to verify that $L^{2}%
(\Omega;\Theta)$ is convex and for each $\lambda \in \lbrack0,1]$, $\xi_{1}$,
$\xi_{2}\in L^{2}(\Omega;\Theta)$,%
\[
E_{P}\left[  \left \vert \frac{1}{n}\sum_{i=1}^{n}X_{i}-(\lambda \xi
_{1}+(1-\lambda)\xi_{2})\right \vert ^{2}\right]  \leq \lambda E_{P}\left[
\left \vert \frac{1}{n}\sum_{i=1}^{n}X_{i}-\xi_{1}\right \vert ^{2}\right]
+(1-\lambda)E_{P}\left[  \left \vert \frac{1}{n}\sum_{i=1}^{n}X_{i}-\xi
_{2}\right \vert ^{2}\right]  .
\]
Note that $\mathcal{P}$ is weakly compact. Then, by Sion's minimax theorem, we
obtain%
\begin{equation}
\inf_{\xi \in L^{2}(\Omega;\Theta)}\sup_{P\in \mathcal{P}}E_{P}\left[
\left \vert \frac{1}{n}\sum_{i=1}^{n}X_{i}-\xi \right \vert ^{2}\right]
=\sup_{P\in \mathcal{P}}\inf_{\xi \in L^{2}(\Omega;\Theta)}E_{P}\left[
\left \vert \frac{1}{n}\sum_{i=1}^{n}X_{i}-\xi \right \vert ^{2}\right]  .
\label{new-e2-8}%
\end{equation}
Since $C_{b}(\Omega)\subset L^{2}(\Omega)$ and $\frac{1}{n}\sum_{i=1}^{n}%
E_{P}[X_{i}|\mathcal{F}_{i-1}]\in \Theta$, $P$-a.s., by Lusin's theorem, we can
find a sequence $\{ \xi_{k}\}_{k=1}^{\infty}$ in $L^{2}(\Omega;\Theta)$ such
that%
\begin{equation}
\lim_{k\rightarrow \infty}E_{P}\left[  \left \vert \frac{1}{n}\sum_{i=1}%
^{n}X_{i}-\xi_{k}\right \vert ^{2}\right]  =E_{P}\left[  \left \vert \frac{1}%
{n}\sum_{i=1}^{n}X_{i}-\frac{1}{n}\sum_{i=1}^{n}E_{P}[X_{i}|\mathcal{F}%
_{i-1}]\right \vert ^{2}\right]  . \label{new-e2-9}%
\end{equation}
Thus we get (\ref{new-e2-10}) by (\ref{new-e2-7})-(\ref{new-e2-9}).
\end{proof}

\end{document}